\documentclass{amsart}
\usepackage{amsrefs}
\usepackage[T1]{fontenc}
\usepackage{graphicx} % Required for inserting images
\usepackage{amsfonts, amsmath, amssymb}
\usepackage[english]{babel}
\usepackage{framed}
\usepackage{array}
\usepackage{listings}
\numberwithin{equation}{section}
\newtheorem{theorem}{Theorem}
\newtheorem{lemma}{Lemma}
\newtheorem{conj}{Conjecture}

\newcommand{\eps}{\varepsilon}
\newcommand{\Z}{\mathbb Z}
\newcommand{\N}{\mathbb N}
\newcommand{\R}{\mathbb R}

\newcommand{\Q}{\mathbb Q}
\title{Evaluations of Series Using the Beta Function}
\author{Roman Le Lan}
\email{roman.lelan1@gmail.com}
\begin{document}
\begin{abstract}
    This article presents proofs of seven conjectural series identities for $\pi$ proposed by Z.-W. Sun. Our proofs rely mainly on transformations involving the Beta function. We also use symbolic integration with the software \texttt{Mathematica} to reduce the amount of computational work.
\end{abstract}
\maketitle
\addtocontents{toc}{\setcounter{tocdepth}{1}} 
\tableofcontents
\section{Introduction}
This article is concerned with series of the form
\[\sum_{k=1}^{\infty}\frac{(ak^2+bk+c)m^k}{k(3k-1)(3k-2)\binom{4k}{k}},\]
\[\sum_{k=1}^\infty \frac{(ak^2+bk+c)m^k\binom{4k}{2k}}{k(4k-1)(4k-3)\binom{3k}{k}\binom{6k}{k}}\]
and
\[\sum_{k=1}^\infty \frac{(ak+b)m^k\binom{2k}{k}}{k(2k-1)\binom{3k}{k}\binom{6k}{k}}\]
with $a,b,c,m\in\Q$. Several evaluations of such series are already known. For example, Au \cite{Au} proved that
\begin{equation*}
\sum_{k=1}^\infty \frac{(35k^2-29k+6)3^k}{k(3k-1)(3k-2)\binom{4k}{k}}=\pi\sqrt3.
\end{equation*}
Similarly, Chu and Zhang \cite{5f4} proved that
\begin{equation*}
\sum_{k=1}^\infty \frac{(22k^2-17k+3)16^k\binom{4k}{2k}}{k(4k-1)(4k-3)\binom{3k}{k}\binom{6k}{k}}=2\pi
\end{equation*}
and that
\begin{equation*}
\sum_{k=1}^\infty \frac{(7k-1)(-4)^k\binom{2k}{k}}{k(2k-1)\binom{3k}{k}\binom{6k}{k}}=-\frac{\pi}{4}.
\end{equation*}
Sun \cite{seriesharmo} \cite{evaltypeseries} \cite{series2} conjectured a large number of similar identities. The aim of this article is to provide a simple method for proving these conjectures.
\begin{theorem}
\begin{itemize}
    \item[(i)] We have
    \begin{equation}
    \sum_{k=1}^\infty  \frac{(13k^2-15k+2)9^{k-1}}{k(3k-1)(3k-2)\binom{4k}{k}}=\frac{\pi}{\sqrt3}.
    \label{th1.1}
\end{equation}
    \item[(ii)] Furthermore, we have the following identities:
    \begin{equation}
    \sum_{k=1}^\infty\frac{(40k^2-20k+3)2^k\binom{4k}{2k}}{k(4k-1)(4k-3)\binom{3k}{k}\binom{6k}{3k}}=\frac\pi2,
    \label{th1.2}
    \end{equation}
    \begin{equation}
    \sum_{k=1}^\infty\frac{(64k^2-48k+7)3^k\binom{4k}{2k}}{k(4k-1)(4k-3)\binom{3k}{k}\binom{6k}{3k}}=\frac{4\pi}{3\sqrt3}.
    \label{th1.3}
    \end{equation}
    \item[(iii)] Finally, we have
    \begin{equation}
    \sum_{k=1}^\infty\frac{(10k-1)(-27)^k\binom{2k}{k}}{k(2k-1)\binom{3k}{k}\binom{6k}{3k}}=-\frac{4\pi}{\sqrt3},
    \label{th1.4}
    \end{equation}
    \begin{equation}
    \sum_{k=1}^\infty\frac{(130k-21)\binom{2k}{k}}{k(2k-1)(-3)^k\binom{3k}{k}\binom{6k}{3k}}=-\frac{2\pi}{3\sqrt3}.
    \label{th1.5}
    \end{equation}
\end{itemize}
\end{theorem}
These results respectively prove Sun's conjectures \cite{series2}*{(4.20)}, \cite{seriesharmo}*{(4.6),(4.8)}, and \cite{evaltypeseries}*{(6.23),(6.20)}. We note that Sun and Zhou \cite{sun-zhou} proved a large number of similar conjectures involving $\binom{2k}{k}/\binom{3k}{k}\binom{6k}{3k}$ by using \textit{cyclotomic multiple zeta values}.

Let $H_n:=\sum_{k=1}^n1/k\quad(n\in\N)$ denote the $n$-th harmonic number. We are also interested in series analogous to those in Theorem 1 that involve harmonic numbers. Define
\[K:=L(2,\chi_{-3})=\sum_{n=0}^\infty\left(\frac{1}{(3n+1)^2}-\frac{1}{(3n+2)^2}\right)\]
and
\[G:=\sum_{n=0}^\infty\frac{(-1)^n}{(2n+1)^2}.\]
The constant $G$ is Catalan's constant. Calegari, Dimitrov, and Tang \cite{cdt} proved that the constant $K$ is irrational.
\begin{theorem}
\begin{itemize}
    \item[(i)] Let $P(X)=95X^2-84X+16$. Then,
    \begin{equation}
    \sum_{k=1}^\infty  \frac{P(k)(H_{4k-1}-H_{k-1})-25k+12}{k(3k-1)(3k-2)\binom{4k}{k}}\left(\frac98\right)^{k-1}=\frac{\pi}{\sqrt3}\log{3}+\frac{15}{4}K.
    \label{th2.1}
    \end{equation}
    \item[(ii)] For every integer $k>0$, let $\mathcal{H}(k):=2H_{6k-1}-H_{3k-1}-H_{k-1}$. Then,
    \begin{equation}
    \sum_{k=1}^\infty\frac{2^k\binom{4k}{2k}[(40k^2-20k+3)\mathcal{H}(k)-32k+4]}{k(4k-1)(4k-3)\binom{3k}{k}\binom{6k}{3k}}=2G+\frac{\pi}{2}\log2.
    \label{th2.2}
    \end{equation}
\end{itemize}
\end{theorem}
These results respectively prove Sun's conjectures \cite{series2}*{(4.19)} and \cite{seriesharmo}*{(4.7)}.
Hou and Sun \cite{hou-sun} also proved several results of this type using Taylor coefficients.

The proofs in this article rely on the Beta function, defined for all $x,y>0$ by:
\[B(x,y)=\frac{\Gamma(x)\Gamma(y)}{\Gamma(x+y)}=\int_0^1t^{x-1}(1-t)^{y-1}\,dt.\]
Here, $\Gamma(z)=\int_0^{+\infty}t^{z-1}e^{-t}\,dt$ is the usual Gamma function. We write $\psi(z)=\frac{d}{dz}\log\Gamma(z)$ for the digamma function.

Finally, we symbolically evaluate integrals of the form:
\[\int_0^1R(t)\,dt\quad(R\in \Z(X)\; \text{a rational function})\]
using the \texttt{IntegrateRational} resource function in the software \texttt{Mathematica}. This is an application of the Lazard-Rioboo-Trager algorithm. We then execute the following few lines of code:
\begin{lstlisting}
R[t_] := (*Rational function to be integrated*);
A[t_] = ResourceFunction["IntegrateRational"][R[t], t];
FullSimplify[A[1] - A[0]]
\end{lstlisting}

This article uses standard notation.
\section{Preliminary Results}
In this section, we present some preliminary results. The first result introduces two polynomial decompositions that will be useful in what follows.
\begin{lemma}
    Let $P$ be a polynomial of degree at most $2$. Then, there exists a unique triple $(\lambda,\mu,\xi)\in\R^3$ such that:
    \begin{equation}
    P(x)=\lambda(3x-1)(3x-2)+\mu x(x+1)+\xi x(4x-1).
    \label{lem1.1}
    \end{equation}
    Moreover, there exists a unique triple $(\alpha,\beta,\gamma)\in\R^3$ such that:
    \begin{equation}
    P(x)=\alpha x(4x-1)+\beta x(4x-3)+\gamma(4x-1)(4x-3).
    \label{lem1.2}
    \end{equation}
\end{lemma}
\begin{proof}
    Consider the vector space $\R_2[X]$ of polynomials of degree at most $2$. It is straightforward to verify that the polynomials $u=(3X-1)(3X-2)$, $v=X(X+1)$, and $w=X(4X-1)$ are linearly independent in $\R_2[X]$. Since $\dim\R_2[X]=3$, the family $(u,v,w)$ forms a basis of $R_2[x]$. Thus, every polynomial $P\in\R_2[X]$ can be written uniquely as a linear combination of $u$, $v$, and $w$.

    The proof of \eqref{lem1.2} is analogous.
\end{proof}
The triples $(\lambda,\mu,\xi)$ and $(\alpha,\beta,\gamma)$ can be readily determined by solving a linear system.

The following lemmas present transformations of binomial coefficients involving the Beta function.

\begin{lemma}
    For every integer $k>0$,
    \begin{equation}
\frac{1}{k\binom{4k}{k}}
=\frac34\int_0^1 t^{3k-1}(1-t)^{k-1}\,dt
\label{lem.eq1}
\end{equation}

\begin{equation}
\frac{k+1}{(3k-1)(3k-2)\binom{4k}{k}}
=\frac34\int_0^1 t^{3(k-1)}(1-t)^{k+1}\,dt
\label{lem.eq2}
\end{equation}

\begin{equation}
\frac{4k-1}{(3k-1)(3k-2)\binom{4k}{k}}
=\frac34\int_0^1 t^{3k-1}(1-t)^{k}\,dt
\label{lem.eq3}
\end{equation}

\begin{equation}
\frac{H_{4k-1}-H_{k-1}}{k\binom{4k}{k}}
=\int_0^1 t^{3k-1}(1-t)^{k-1}\log(1-t)\,dt
\label{lem.eq4}
\end{equation}

\begin{equation}
\frac{(k+1)(H_{4k-1}-H_{k+1})}
{(3k-1)(3k-2)\binom{4k}{k}}
=\int_0^1 t^{3(k-1)}(1-t)^{k+1}\log(1-t)\,dt
\label{lem.eq5}
\end{equation}

\begin{equation}
\frac{(4k-1)(H_{4k-2}-H_k)}
{(3k-1)(3k-2)\binom{4k}{k}}
=\int_0^1 t^{3(k-1)}(1-t)^k\log(1-t)\,dt
\label{lem.eq6}
\end{equation}
\end{lemma}
\begin{proof}
    These identities are readily established using the Beta function and its derivative by noting that $\psi(n)=H_{n-1}-\gamma$. Here, $\gamma$ denotes the Euler's constant.
\end{proof}
\begin{lemma}
    Let $k>0$ be an integer. We have the following identities.
    \begin{equation}
        \frac{\binom{4k}{2k}}{2k\binom{3k}{k}\binom{6k}{3k}}=\frac{1}{4^k}\int_0^1t^{4k}(1-t^2)^{k-1}\,dt
    \label{id.a}
    \end{equation}
    \begin{equation}
        \frac{\binom{4k}{2k}}{(4k-1)\binom{3k}{k}\binom{6k}{3k}}=\frac{1}{4^k}\int_0^1t^{4k-2}(1-t^2)^k\,dt
        \label{id.b}
    \end{equation}
    \begin{equation}
        \frac{7\binom{4k}{2k}}{(4k-3)\binom{3k}{k}\binom{6k}{3k}}=\frac{1}{4^k}\int_0^1(t^2+4)t^{4(k-1)}(1-t^2)^k\,dt
        \label{id.c}
    \end{equation}
    \begin{equation}
        \frac{\binom{4k}{2k}(2H_{6k-1}-H_{3k-1}-H_{k-1})}{2k\binom{3k}{k}\binom{6k}{3k}}=\frac{1}{4^k}\int_0^1t^{4k}(1-t^2)^{k-1}\log\left(\frac4{1-t^2}\right)\,dt
    \label{id.d}
    \end{equation}
    \begin{equation}
        \frac{\binom{4k}{2k}(2H_{6k-1}-H_{3k-1}-H_k)}{(4k-1)\binom{3k}{k}\binom{6k}{3k}}=\frac{1}{4^k}\int_0^1t^{4k-2}(1-t^2)^k\log\left(\frac4{1-t^2}\right)\,dt
        \label{id.e}
    \end{equation}
    \begin{multline}
        \frac{7\binom{4k}{2k}\left(2H_{6k-1}-H_{3k-1}-H_k-\frac{8}{7(4k-1)}\right)}{(4k-3)\binom{3k}{k}\binom{6k}{3k}}\\=\frac{1}{4^k}\int_0^1(t^2+4)t^{4(k-1)}(1-t^2)^k\log\left(\frac4{1-t^2}\right)\,dt
        \label{id.f}
    \end{multline}
\end{lemma}
\begin{proof}
    \textbf{Proof of \eqref{id.a}-\eqref{id.c}.} Let $k>0$ be an integer. By making the change of variable $x=t^2$ in the integral \eqref{id.a}, we obtain
    \begin{align*}
    \frac{1}{4^k}\int_0^1t^{4k}(1-t^2)^{k-1}\,dt&=\frac{1}{2\cdot4^{k}}\int_0^1x^{2k-1/2}(1-x)^{k-1}\,dx\\&=\frac{1}{2^{2k+1}}B\left(2k+\frac12,k\right).        
    \end{align*}
    Then, using Legendre's duplication formula for the $\Gamma$ function and simplifying, we obtain
    \begin{align*}
     \frac{1}{4^k}\int_0^1t^{4k}(1-t^2)^{k-1}\,dt&=\frac{\Gamma(4k)\Gamma(k)\Gamma(3k)}{2\Gamma(2k)\Gamma(6k)}\\&=\frac{\binom{4k}{2k}}{2k\binom{3k}{k}\binom{6k}{3k}}.
    \end{align*}
    This proves \eqref{id.a}. We proceed in the same way to prove \eqref{id.b}. Moreover, observing that
    \[\frac{1}{2\cdot4^{k}}B\left(2k-\frac32,k+2\right)=\frac{2(k+1)\binom{4k}{2k}}{(4k-1)(4k-3)\binom{3k}{k}\binom{6k}{3k}}\]
    and that
    \[\frac{\binom{4k}{2k}}{(4k-1)\binom{3k}{k}\binom{6k}{3k}}=\frac{(4k-3)\binom{4k}{2k}}{(4k-1)(4k-3)\binom{3k}{k}\binom{6k}{3k}},\]
    we readily deduce \eqref{id.c}.

    \textbf{Proof of \eqref{id.d}-\eqref{id.f}.} Let $x,y>0$.
    Then,
    \begin{align*}
    \frac{\partial}{\partial y}B(x,y)&=\int_0^1u^{x-1}(1-u)^{y-1}\log(1-u)\,du\\&=\frac{\Gamma(x)\Gamma(y)}{\Gamma(x+y)}(\psi(y)-\psi(x+y)).
    \end{align*}
    Thus,
    \[\frac{1}{4^k}\int_0^1u^{2k-1/2}(1-u)^{k-1}\log(1-u)\,du=\frac{1}{4^k}B\left(2k+\frac12,k\right)\left[\psi(k)-\psi\left(3k+\frac12\right)\right]\]
    Using Legendre's duplication formula for the $\Gamma$ and $\psi$ functions, together with the fact that $\psi(n)=H_{n-1}-\gamma$, where $\gamma$ is the Euler-Mascheroni constant, we obtain:
    \begin{multline*}
        \frac{1}{4^k}\int_0^1u^{2k-1/2}(1-u)^{k-1}\log(1-u)\,du\\=\frac{\Gamma(4k)\Gamma(k)\Gamma(3k)}{2\Gamma(2k)\Gamma(6k)}(-H_{6k-1}+H_{3k-1}+H_{k-1}+\log{4}).
    \end{multline*}
    We deduce \eqref{id.d}. The proof of \eqref{id.e} and \eqref{id.f} proceeds in essentially the same way.
\end{proof}
The following identities are due to Sun and Zhou.
\begin{lemma}[\cite{sun-zhou}*{Lemma 2.1}]
For every integer $k>0$,
    \begin{equation}
    \frac{\binom{2k}{k}}{4k\binom{3k}{k}\binom{6k}{3k}}=\frac{1}{16^k}\int_0^1t^{2k}(1-t^2)^{2k-1}\,dt,
    \label{lem2.1}
    \end{equation}
    \begin{equation}
    \frac{\binom{2k}{k}}{(2(k-1))\binom{3k}{k}\binom{6k}{3k}}=\frac{1}{16^k}\int_0^1t^{2k-1}(1-t^2)^{2k}\,dt.
    \label{lem2.2}
    \end{equation}
    Moreover,
    \begin{multline}
    \frac{\left(2H_{6k}-H_{3k}-H_k+\frac1k-4\log2\right)\binom{2k}{k}}{8k\binom{3k}{k}\binom{6k}{3k}}\\=\frac{1}{16^k}\int_0^1t^{2k}(1-t^2)^{2k-1}\log\left(\frac{t}{1-t^2}\right)\,dt,
    \label{lem2.3}
    \end{multline}
    \begin{multline}
    \frac{\left(H_{2k}-H_k+\frac1{2k}-2\log2\right)\binom{2k}{k}}{4k\binom{3k}{k}\binom{6k}{3k}}\\=\frac{1}{16^k}\int_0^1t^{2k}(1-t^2)^{2k-1}\log\left(\frac{t^2}{1-t^2}\right)\,dt,
    \label{lem2.4}
    \end{multline}
    \begin{multline}
    \frac{\left(2H_{6k}-H_{3k}-H_k-\frac{2}{2k-1}-4\log2\right)\binom{2k}{k}}{2(2k-1)\binom{3k}{k}\binom{6k}{3k}}\\=\frac{1}{16^k}\int_0^1t^{2k-1}(1-t^2)^{2k-1}\log\left(\frac{t}{1-t^2}\right)\,dt,
    \label{lem2.5}
    \end{multline}
    \begin{multline}
    \frac{\left(H_{2k}-H_k-\frac2{2k-1}-2\log2\right)\binom{2k}{k}}{(2k-1)\binom{3k}{k}\binom{6k}{3k}}\\=\frac{1}{16^k}\int_0^1t^{2k-1}(1-t^2)^{2k}\log\left(\frac{t^2}{1-t^2}\right)\,dt.
    \label{lem2.6}
    \end{multline}
\end{lemma}
\begin{proof}
    The identities \eqref{lem2.3}-\eqref{lem2.6} can be readily derived by combining the identities given by Sun and Zhou.
\end{proof}
The final lemma gives the evaluation of an integral that will be needed in the proof of \eqref{th2.1}.
\begin{lemma}
We have the following identity:
    \begin{equation}
    \int_0^1\frac{\log(1-t)}{3t^2-6t+4}\,dt=-\frac{\pi\sqrt{3}}{18}\log3-\frac{5}{8}K.
    \label{lem3}
    \end{equation}
\end{lemma}
\begin{proof}
    Let
    \[I=\int_0^1\frac{\log(1-t)}{3t^2-6t+4}\,dt.\]
    Make the change of variable $t=1-x$. Then,
    \[I=\int_0^1\frac{\log(x)}{3x^2+1}\,dx.\]
    Now, making the change of variable $x\sqrt3=\tan\theta$, we obtain:
    \begin{align}
        I&=\frac{\sqrt3}{3}\int_0^{\pi/3}\log\left(\frac{\tan\theta}{\sqrt3}\right)\,d\theta\nonumber\\
        &=-\frac{\pi\sqrt{3}}{18}\log3+\frac{\sqrt3}{3}\left(\int_0^{\pi/3}\log\left(\sin\theta\right)\,d\theta-\int_0^{\pi/3}\log\left(\cos\theta\right)\,d\theta\right)\nonumber\\
        &:=-\frac{\pi\sqrt{3}}{18}\log3+\frac{\sqrt3}3(J_1-J_2)
        \label{I=J1-J2}
    \end{align}
    where $J_1$ and $J_2$ are the respective integrals appearing in the second line. Recall the integral representation of the Clausen function:
    \[\text{Cl}_2(z)=\int_0^z\log\left|2\sin\frac{\theta}{2}\right|\,d\theta.\]
    Then, making the change of variable $\theta=u/2$ in the expression for $J_1$, it is straightforward to verify that:
    \begin{equation}
        J_1=-\frac12\text{Cl}_2(2\pi/3)-\frac\pi3\log2.
    \label{J1}
    \end{equation}
    Furthermore, making the change of variable $u=\pi/2-\theta$ in the expression for $J_2$, we obtain:
    \[J_2=\int_{\pi/6}^{\pi/2}\log{(\sin u)}\,du.\]
    This can be evaluated in the same way as $J_1$. Thus, after a few simplifications,
    \begin{equation}
    J_2=\frac12\text{Cl}_2(\pi/3)-\frac\pi3\log2
    \label{J2}
    \end{equation}
    Since \[\text{Cl}_2(\pi/3)=\frac32\text{Cl}_2(2\pi/3)=\frac{3\sqrt3}{4}K,\]
    substituting \eqref{J1} and \eqref{J2} into \eqref{I=J1-J2}, we obtain:
    \begin{align*}
        I&=-\frac{\pi\sqrt{3}}{18}\log3-\frac{\sqrt3}3\cdot\frac{5\sqrt3}{8}K\\&=-\frac{\pi\sqrt{3}}{18}\log3-\frac{5}{8}K.
    \end{align*}
    This proves Lemma 5.
\end{proof}
\section{Proof of Theorem 1 (i)}
Let $a,b,c,m\in\R$. Define
\[S=\sum_{k=1}^\infty\frac{(ak^2+bk+c)m^k}{k(3k-1)(3k-2)\binom{4k}{k}}.\]
Then, by equation \eqref{lem1.1} of Lemma 1, there exists $(\lambda,\mu,\xi)\in\R^3$ such that:
\[S=\lambda\sum_{k=1}^\infty \frac{m^k}{k\binom{4k}{k}}+\mu\sum_{k=1}^\infty \frac{(k+1)m^k}{(3k-1)(3k-2)\binom{4k}{k}}+\xi\sum_{k=1}^\infty \frac{(4k-1)m^k}{(3k-1)(3k-2)\binom{4k}{k}}.\]
By equations \eqref{lem.eq1}-\eqref{lem.eq3} of Lemma 2, we can therefore write:
\begin{multline*}
\frac43S=\lambda\int_0^1\sum_{k=1}^\infty t^{3k-1}(1-t)^{k-1}m^k\,dt+\mu\int_0^1\sum_{k=1}^\infty t^{3(k-1)}(1-t)^{k+1}m^k\,dt\\+\xi\int_0^1\sum_{k=1}^\infty t^{3(k-1)}(1-t)^{k}m^k\,dt.
\end{multline*}
Recognizing geometric series and simplifying, we obtain:
\[S=\frac34\int_0^1\frac{t^2(\lambda+\mu)-t(2\mu+\xi)+\mu+\xi}{t^4-t^3+1/m}\,dt.\]
\subsection{Proof of \eqref{th1.1}}
The triple $(\lambda,\mu,\xi)$ associated with the polynomial $13X^2-15X+2$ is $(1,-4,2)$. Thus, using \texttt{IntegrateRational}, we obtain
\[\sum_{k=1}^\infty  \frac{(13k^2-15k+2)9^{k-1}}{k(3k-1)(3k-2)\binom{4k}{k}}=\frac34\int_0^1\frac{-3t^2+6t-2}{9t^4-9t^3+1}\,dt=\frac{\pi}{\sqrt3}.\]

%%%%%%%%%%%%%%%%%%%%%%%%%%%%%%%%%%%%%%%%%%%%%%%%%%%%%%%%%%%%%%%%%%%%%%%%%
\section{Proof of Theorem 1 (ii)}
Let $a,b,c,m\in\R$. Define
\[S=\sum_{k=1}^\infty\frac{(ak^2+bk+c)m^k\binom{4k}{2k}}{k(4k-1)(4k-3)\binom{3k}{k}\binom{6k}{3k}}.\]
Then, by equation \eqref{lem1.2} of Lemma 1, there exists $(\alpha,\beta,\gamma)\in\R^3$ such that:
\[S=\alpha\sum_{k=1}^\infty\frac{m^k\binom{4k}{2k}}{(4k-3)\binom{3k}{k}\binom{6k}{3k}}+\beta\sum_{k=1}^\infty\frac{m^k\binom{4k}{2k}}{(4k-1)\binom{3k}{k}\binom{6k}{3k}}+\alpha\sum_{k=1}^\infty\frac{m^k\binom{4k}{2k}}{k\binom{3k}{k}\binom{6k}{3k}}.\]
By Lemma 3, we can therefore write:
\begin{multline*}
S=\frac17\alpha\int_0^1\sum_{k=1}^\infty\frac{m^k}{4^k}(t^2+4)t^{4(k-1)}(1-t^2)^k\,dt+\beta\int_0^1\sum_{k=1}^\infty\frac{m^k}{4^k}t^{4k-2}(1-t^2)^k\,dt\\+2\gamma\int_0^1\sum_{k=1}^\infty\frac{m^k}{4^k}t^{4k}(1-t^2)^{k-1}\,dt
\end{multline*}
Recognizing geometric series and simplifying, we obtain:
\[S=\frac17\int_0^1\frac{4\alpha+(-3\alpha+7\beta)t^2
+(-\alpha-7\beta+14\gamma)t^4}{4/m-t^4+t^6}\,dt.\]
\subsection{Proof of \eqref{th1.2}}
The triple $(\alpha,\beta,\gamma)$ associated with the polynomial $40X^2-20X+3$ is $(7,-1,1)$. Thus, using \texttt{IntegrateRational}, we obtain
\[\sum_{k=1}^\infty\frac{(40k^2-20k+3)2^k\binom{4k}{2k}}{k(4k-1)(4k-3)\binom{3k}{k}\binom{6k}{3k}}=2\int_0^1\frac{2-2t^2+t^4}{2-t^4+t^6}\,dt=\frac\pi2.\]
\subsection{Proof of \eqref{th1.3}}
The triple $(\alpha,\beta,\gamma)$ associated with the polynomial $3(64X^2-48X+7)$ is $(14,6,7)$. Thus, using \texttt{IntegrateRational}, we obtain
\[\sum_{k=1}^\infty\frac{(64k^2-48k+7)3^k\binom{4k}{2k}}{k(4k-1)(4k-3)\binom{3k}{k}\binom{6k}{3k}}=\frac23\int_0^1\frac{4+3t^4}{4/3-t^4+t^6}\,dt=\frac{4\pi}{3\sqrt3}.\]
\section{Proof of Theorem 1 (iii)}
Let $a,b,m\in\R$. Define:
\[S=\sum_{k=1}^\infty\frac{(ak+b)m^k\binom{2k}{k}}{k(2k-1)\binom{3k}{k}\binom{6k}{3k}}\]
By identities \eqref{lem2.1} and \eqref{lem2.2} of Lemma 3,
\[S=(a+2b)\int_0^1\sum_{k=1}^\infty\frac{m^k}{16^k}t^{2(k-1)}(1-t^2)^{2k}\,dt-b\int_0^1\sum_{k=1}^\infty\frac{4m^k}{16^k}t^{2k}(1-t^2)^{2k-1}\,dt\]
Using hypergeometric series, we obtain:
\[S=\int_0^1\frac{(t^2-1)\left(-a-2b+(a+6b)t^2\right)}{16/m-t^2(t^2-1)^2}\,dt.
\]
\subsection{Proof of \eqref{th1.4}}
Using \texttt{IntegrateRational}, we obtain:
\[\sum_{k=1}^\infty\frac{(10k-1)(-27)^k\binom{2k}{k}}{k(2k-1)\binom{3k}{k}\binom{6k}{3k}}=\int_0^1\frac{(t^2-1)(-8+4t^2)}{-16/27-t^2(t^2-1)^2}=-\frac{4\pi}{\sqrt3}.\]
\subsection{Proof of \eqref{th1.5}}
Using \texttt{IntegrateRational}, we obtain:
\[\sum_{k=1}^\infty\frac{(130k-21)\binom{2k}{k}}{k(2k-1)(-3)^k\binom{3k}{k}\binom{6k}{3k}}=\int_0^1\frac{(t^2-1)(-88+4t^2)}{-48-t^2(t^2-1)^2}\,dt=-\frac{2\pi}{3\sqrt3}.\]
%%%%%%%%%%%%%%%%%%%%%%%%%%%%%%%%%%%%%%%%%%%%%%%%%%%%%%%%%%%%%%%%
\section{Proof of Theorem 2 (i)}
%%%%%%%%%%%%%%%%%%%%
\begin{proof}
    Let $P(X)=95X^2-84X+16$ be a polynomial. We have
    \begin{equation}
    S:=\sum_{k=1}^\infty  \frac{P(k)(H_{4k-1}-H_{k-1})-25k+12}{k(3k-1)(3k-2)\binom{4k}{k}}\left(\frac98\right)^{k-1}:=\frac89(S_1+S_2)
    \label{S}
    \end{equation}
    where
    \[S_1:=\sum_{k=1}^\infty  \frac{P(k)(H_{4k-1}-H_{k-1})}{k(3k-1)(3k-2)\binom{4k}{k}}\left(\frac98\right)^k\]
    and
    \[S_2:=\sum_{k=1}^\infty  \frac{-25k+12}{k(3k-1)(3k-2)\binom{4k}{k}}\left(\frac98\right)^k\]
    Observe that the triple $(\lambda,\mu,\xi)$ from Lemma 1 associated with $P$ is $(8,-5,7)$. Thus,
    \begin{multline*}
        S_1=8\sum_{k=1}^\infty \frac{H_{4k-1}-H_{k-1}}{k\binom{4k}{k}}\left(\frac98\right)^k-5\sum_{k=1}^\infty\frac{(k+1)(H_{4k-1}-H_{k+1})}{(3k-1)(3k-2)\binom{4k}{k}}\left(\frac98\right)^k\\-5\sum_{k=1}^\infty\frac{(k+1)\left(\frac1k+\frac1{k+1}\right)}{(3k-1)(3k-2)\binom{4k}{k}}\left(\frac98\right)^k+7\sum_{k=1}^\infty\frac{(4k-1)(H_{4k-2}-H_{k})}{(3k-1)(3k-2)\binom{4k}{k}}\left(\frac98\right)^k\\+7\sum_{k=1}^\infty\frac{(4k-1)\left(\frac1{4k-1}+\frac1k\right)}{(3k-1)(3k-2)\binom{4k}{k}}\left(\frac98\right)^k.
    \end{multline*}
    Note that
    \[5\sum_{k=1}^\infty\frac{(k+1)\left(\frac1k+\frac1{k+1}\right)}{(3k-1)(3k-2)\binom{4k}{k}}\left(\frac98\right)^k-7\sum_{k=1}^\infty\frac{(4k-1)\left(\frac1{4k-1}+\frac1k\right)}{(3k-1)(3k-2)\binom{4k}{k}}\left(\frac98\right)^k=S_2\]
    By identities \eqref{lem.eq4}-\eqref{lem.eq6} of Lemma 2, we have:
    \begin{align*}
        -\frac43(S_1+S_2)=&8\int_0^1\sum_{k=1}^\infty t^{3k-1}(1-t)^{k-1}\left(\frac98\right)^k\log{(1-t)}\,dt\\&-5\int_0^1\sum_{k=1}^\infty t^{3(k-1)}(1-t)^{k+1}\left(\frac98\right)^k\log(1-t)\,dt\\&+7\int_0^1\sum_{k=1}^\infty t^{3(k-1)}(1-t)^k\left(\frac98\right)^k\log(1-t)\,dt
        \\=&\frac98\Bigg(8\int_0^1\frac{t^2}{1-\frac98t^3(1-t)}\log(1-t)\,dt\\&-5\int_0^1\frac{(1-t)^2}{1-\frac98t^3(1-t)}\log(1-t)\,dt\\&+7\int_0^1\frac{1-t}{1-\frac98t^3(1-t)}\log(1-t)\,dt\Bigg)\\=&\int_0^1\frac{9(3t^2+3t+2)}{9t^4-9t^3+8}\log(1-t)\,dt
    \end{align*}
    Observe that $9X^4-9X^3+8=(3X^2+3X+2)(3X^2-6X+4)$. Hence,
    \[-\frac4{27}(S_1+S_2)=\int_0^1\frac{\log(1-t)}{3t^2-6t+4}\,dt.\]
   By Lemma 5:
    \begin{equation}
    S_1+S_2=\frac{3\pi\sqrt{3}}{8}\log3+\frac{135}{32}K.
    \label{S1}
    \end{equation}
    Substituting \eqref{S1} into \eqref{S}, we obtain the desired result:
    \[S=\frac{\pi\sqrt3}3\log{3}+\frac{15}{4}K.\]
\end{proof}
This proves identity \eqref{th2.1}.
\section{Proof of Theorem 2 (ii)}
\begin{proof}
For every integer $n>0$, let $\mathcal{H}(n):=2H_{6n-1}-H_{3n-1}-H_{n-1}$. Let
\[S:=\sum_{k=1}^\infty\frac{2^k\binom{4k}{2k}[(40k^2-20k+3)\mathcal{H}(k)-32k+4]}{k(4k-1)(4k-3)\binom{3k}{k}\binom{6k}{3k}}.\]
The triple $(\lambda,\mu,\xi)$ from Lemma 1 associated with the polynomial $40X^2-20X+3$ is $(7,-1,1)$. We therefore obtain:
\begin{multline*}
    S=7\sum_{k=1}^\infty\frac{2^k\binom{4k}{2k}\mathcal{H}(k)}{(4k-3)\binom{3k}{k}\binom{6k}{3k}}-\sum_{k=1}^\infty\frac{2^k\binom{4k}{2k}\mathcal{H}(k)}{(4k-1)\binom{3k}{k}\binom{6k}{3k}}+\sum_{k=1}^\infty\frac{2^k\binom{4k}{2k}\mathcal{H}(k)}{k\binom{3k}{k}\binom{6k}{3k}}\\+\sum_{k=1}^\infty\frac{2^k\binom{4k}{2k}(-32k+4)}{k(4k-1)(4k-3)\binom{3k}{k}\binom{6k}{3k}}.
\end{multline*}
Substituting identities \eqref{id.d}, \eqref{id.e}, and \eqref{id.f} from Lemma 3 above, and then simplifying, we obtain
\begin{multline}
    S=\int_0^1\sum_{k=1}^\infty\frac{1}{2^k}(t^2+4)t^{4(k-1)}(1-t^2)^k\log\left(\frac4{1-t^2}\right)\,dt\\-\int_0^1\sum_{k=1}^\infty\frac{1}{2^k}t^{4k-2}(1-t^2)^k\log\left(\frac4{1-t^2}\right)\,dt\\+2\int_0^1\sum_{k=1}^\infty\frac{1}{2^k}t^{4k}(1-t^2)^{k-1}\log\left(\frac4{1-t^2}\right)\,dt\\+7\sum_{k=1}^\infty\frac{2^k\binom{4k}{2k}\left(\frac1k+\frac{8}{7(4k-1)}\right)}{(4k-3)\binom{3k}{k}\binom{6k}{3k}}\\-\sum_{k=1}^\infty\frac{2^k\binom{4k}{2k}\frac{1}{k}}{(4k-1)\binom{3k}{k}\binom{6k}{3k}}+\sum_{k=1}^\infty\frac{2^k\binom{4k}{2k}(-32k+4)}{k(4k-1)(4k-3)\binom{3k}{k}\binom{6k}{3k}}.
    \label{series3-transform}
\end{multline}
Observe that
\begin{multline*}
    0=7\sum_{k=1}^\infty\frac{2^k\binom{4k}{2k}\left(\frac1k+\frac{8}{7(4k-1)}\right)}{(4k-3)\binom{3k}{k}\binom{6k}{3k}}-\sum_{k=1}^\infty\frac{2^k\binom{4k}{2k}\frac{1}{k}}{(4k-1)\binom{3k}{k}\binom{6k}{3k}}\\+\sum_{k=1}^\infty\frac{2^k\binom{4k}{2k}(-32k+4)}{k(4k-1)(4k-3)\binom{3k}{k}\binom{6k}{3k}}.
\end{multline*}
Then, recognizing geometric series in \eqref{series3-transform} and simplifying, we have
\begin{align*}
\frac12S&=\int_0^1\frac{t^4-2t^2+2}{t^6-t^4+2}\log\left(\frac4{1-t^2}\right)\,dt\\&=\int_0^1\frac{1}{t^2+1}\log\left(\frac4{1-t^2}\right)
\end{align*}
Evaluating the integral above using the \texttt{Integrate} function in \texttt{Mathematica}, we directly obtain
\[S=2G+\frac12\pi\log2.\]
This proves identity \eqref{th2.2}.
\end{proof}
\section{Concluding Remarks}
The method proposed in this article makes it possible to systematically prove identities of a "simple" form (without harmonic numbers), provided that one can find a representation of the terms of the series as a linear combination of Beta functions.

As for harmonic numbers of order 1 (and higher orders), certain problems appear to be inaccessible to our method. For example, we have been unable to prove the following conjecture:
\begin{conj}[Sun \cite{series2}*{(4.25)}]
Let $U(X)=415X^2-343X+62$ and $V(X)=581X^2-229X-6$. Then,
\[\sum_{k=1}^\infty \frac{U(k)-V(k)/(4k)}{k(3k-1)(3k-2)(-8)^k\binom{4k}{k}}=-\frac{\pi^2}{4}.\]
\end{conj}
Indeed, let
\[S=\sum_{k=1}^\infty \frac{P(k)m^k}{Q(k)c_k}\]
with $P,Q\in\Z[X]$, $m\in\Q$, and $c_k=\prod_i\binom{a_ik}{b_ik}^{\eps_i}$ with $a_i,b_i\in\N$ and $\eps_i\in\Z$. Then, our method for evaluating $S$ (or analogues of $S$ involving harmonic numbers) depends heavily on several conditions:
\begin{itemize}
    \item[1] There must exist a factorization $Q=\prod_{i=1}^nD_i$ in $\Z[X]$ such that $D_i(k)\cdot c_k=a\cdot B(x_k,y_k)$, where $a$ is a constant and $\Re(x_k),\Re(y_k)>0$.
    \item[2] In order to apply a decomposition of $P$ analogous to that of Lemma 1, we need $\deg Q\ge \deg P$.
    \item[3] Note that the first condition implies that $Q$ has no roots of multiplicity $\ge2$.
\end{itemize}
These conditions therefore restrict the applicability of our method to certain specific cases.
\bibliographystyle{amsplain}
\bibliography{bibly}
\end{document}